\documentclass[a4paper,12pt]{article}
\usepackage[utf8]{inputenc}
\usepackage{amsfonts}
\usepackage{amssymb}
\usepackage{amsmath}
\usepackage{amsthm}
\usepackage{tikz}
\usepackage{url}

\theoremstyle{plain}
\newtheorem{defi}{Definition}[section]
\newtheorem{thm}{Theorem}[section]
\newtheorem{prop}{Proposition}[section]
\newtheorem{lem}{Lemma}[section]
\newtheorem{coro}{Corollary}[section]
\theoremstyle{remark}

\DeclareMathOperator{\Gal}{Gal}
\DeclareMathOperator{\Aut}{Aut}
\DeclareMathOperator{\Jac}{Jac}
\DeclareMathOperator{\Pic}{Pic}
\DeclareMathOperator{\Spec}{Spec}
\DeclareMathOperator{\Syl}{Syl}
\DeclareMathOperator{\sw}{sw}
\DeclareMathOperator{\Sw}{Sw}

\DeclareMathOperator{\Sl}{Sl}

\DeclareMathOperator{\Hom}{Hom}
\DeclareMathOperator{\codim}{codim}
\DeclareMathOperator{\Id}{Id}
\DeclareMathOperator{\disc}{disc}
\DeclareMathOperator{\Ker}{Ker}

\newcommand{\N}{\mathbb{N}}
\newcommand{\M}{M}
\newcommand{\Z}{\mathbb{Z}}
\newcommand{\R}{\mathbb{R}}
\newcommand{\Q}{\mathbb{Q}}
\newcommand{\F}{\mathbb{F}}
\newcommand{\Qpur}{\mathbb{Q}_p^{\rm ur}}
\newcommand{\Proj}{\mathbb{P}^1}
\renewcommand{\L}{L}

\newcounter{mysub}

\begin{document}
\title{Maximal wild monodromy in unequal characteristic}
\author{P. Chrétien \and M. Matignon}
\maketitle
\begin{abstract}
Let $R$ be a complete discrete valuation ring of mixed characteristic $(0,p)$ with fraction field $K$. We study stable models of $p$-cyclic covers of $\Proj_K$. First, we determine the monodromy extension, the monodromy group, its filtration and the Swan conductor for special covers of arbitrarily high genus with potential good reduction. In the case $p=2$ we consider hyperelliptic curves of genus $2$.
\end{abstract}
\section{Introduction}
\label{sec:0}
Let $(R,v)$ be a complete discrete valuation ring of mixed characteristic $(0,p)$ with fraction field $K$ containing a primitve $p$-th root of unity $\zeta_p$ and algebraically closed residue field $k$. The stable reduction theorem states that given a smooth, projective, geometrically connected curve $C/K$ of genus $g(C) \geq 2$, there exists a unique minimal Galois extension $\M/K$ called \textsl{the monodromy extension of $C/K$} such that $C_{\M}:=C \times \M$ has stable reduction over $\M$. The group $G =  \Gal(\M/K)$ is the \textsl{monodromy group of $C/K$}.
In a previous paper, Lehr and Matignon \cite{LM2} gave an algorithm to determine the stable reduction of $p$-cyclic covers of $\Proj_K$ under the extra assumption of \textsl{equidistant geometry}  of the branch locus and obtain information about the monodromy extension $\M/K$ of $C/K$. This makes effective a theorem of Raynaud \cite{Ra} in the case of $p$-cyclic covers of $\Proj_K$. The present article studies examples of such $p$-cyclic covers but is independent of their work and develops specific methods to treat our special covers. \\
\indent Let $\mathcal{C}$ be the stable model of $C_{\M}/\M$ and $\Aut_k(\mathcal{C}_k)^{\#}$ the subgroup of $\Aut_k(\mathcal{C}_k)$ of elements acting trivially on the reduction in $\mathcal{C}_k$ of the ramification locus of $C_{\M}\to \Proj_{\M}$ (see \cite{Liu} 10.1.3 for the definition of the reduction map of $C_{\M}$). One derives from the stable reduction theorem the following injection :
\begin{equation}\label{injintro}
\Gal(\M/K) \hookrightarrow \Aut_k(\mathcal{C}_k)^{\#}.
\end{equation}
When the $p$-Sylow subgroups of these groups are isomorphic, one says that the \textsl{wild monodromy is maximal}. We are interested in realization of covers such that the $p$-adic valuation of $|\Aut_k(\mathcal{C}_k)^{\#}|$ is large and having maximal wild monodromy, we will study ramification filtrations and Swan conductors of their monodromy extensions.\\

In section \ref{sec:2}, we consider examples of covers of arbitrarily high genus having potential good reduction. Let $n \in \N^{\times}$, $q=p^n$, $\lambda= \zeta_p-1$ and ${K=\Qpur (\lambda^{1/(1+q)})}$. We study covers $C_c/K$ of $\Proj_K$ defined by $Y^p=1+cX^q+X^{1+q}$ with $c \in R$, $ v(\lambda^{p/(1+q)}) > v(c)$ and $v(c^{p}-c) \geq v(p)$.
 
\begin{thm}\label{prop1intro}
The stable reduction $\mathcal{C}_k/k$ is canonically a $p$-cyclic cover of $\Proj_k$. It is smooth, ramified at one point $\infty$ and étale outside $\infty$. The ramification locus of $C_{\M} \to \Proj_{\M}$ reduces in $\infty$ and the group $\Aut_k(\mathcal{C}_k)^{\#}$ has a unique $p$-Sylow subgroup $\Aut_k(\mathcal{C}_k)^{\#}_1$. Moreover, the curve $C_c/K$ has maximal wild monodromy $\M/K$. The extension $\M/K$ is the decomposition field of an explicitly given polynomial and $\Gal(\M/K) \simeq \Aut_k(\mathcal{C}_k)^{\#}_1$ is an extra-special $p$-group of order $pq^2$.\end{thm} 

\indent Let $X/k$ be a $p$-cyclic cover of $\Proj_k$ of genus $g(X)$, ramified at one point $\infty$ and étale outside $\infty$. According to \cite{LM}, the $p$-Sylow subgroup $G_{\infty,1}(X)$ of the subgroup of $\Aut_k(X)$ of automorphisms leaving $\infty$ fixed satisfies $|G_{\infty,1}(X)| \leq \frac{4p}{(p-1)^ 2}g(X)^ 2$. The stable reduction $\mathcal{C}_k/k$ of Theorem \ref{prop1intro} is such that $G_{\infty,1}(\mathcal{C}_k)=\Aut_k(\mathcal{C}_k)^{\#}_1$ and $|G_{\infty,1}(\mathcal{C}_k)|=\frac{4p}{(p-1)^ 2}g(\mathcal{C}_k)^ 2$. So we obtain the largest possible maximal wild monodromy for curves over some finite extension of $\Qpur$ with genus in $\frac{p-1}{2}p^{\N}$ in the good reduction case.\\
\indent The group $G_{\infty,1}(\mathcal{C}_k)=\Aut_k(\mathcal{C}_k)^{\#}_1$ is endowed with the ramification filtration $(G_{\infty,i}(\mathcal{C}_k))_{i\geq 0}$  which is easily seen to be :
\[ G_{\infty,0}(\mathcal{C}_k) = G_{\infty,1}(\mathcal{C}_k)\supsetneq {\rm Z}(G_{\infty,0}(\mathcal{C}_k))= G_{\infty,2}(\mathcal{C}_k) = \dots = G_{\infty,1+q}(\mathcal{C}_k) \supsetneq \lbrace 1 \rbrace. \]
Moreover, $G:=\Gal(\M/K)$ being the Galois group of a finite extension of $\Qpur$, it is endowed with the ramification filtration $(G_i)_{i \geq 0}$ of an arithmetic nature. Since $G \simeq G_{\infty,1}(\mathcal{C}_k)$ it is natural to ask for the behaviour of $(G_i)_{i \geq 0}$ under \eqref{injintro}, that is to compare $(G_i)_{i \geq 0}$ and  $(G_{\infty,i}(\mathcal{C}_k))_{i\geq 0}$. One shows that they actually coincide and we compute the conductor exponent $f(\Jac(C_c)/K)$ of $\Jac(C_c)/K$ and its Swan conductor $\sw(\Jac(C_c)/K)$ :

\begin{thm}
Under the hypotheses of Theorem \ref{prop1intro}, the lower ramification filtration of $G$ is :
\[ G= G_0 = G_1 \supsetneq   {\rm Z}(G)= G_2 = \dots = G_{1+q} \supsetneq \lbrace 1 \rbrace. \]
Then, $f(\Jac(C_c)/K)=(2q+1)(p-1)$ and, in the case where $c \in \Qpur$, $\sw(\Jac(C_c)/\Qpur)=1$.
\end{thm}

The value $\sw(\Jac(C_c)/\Qpur) =1$ is the smallest one among abelian varieties over $\Qpur$ with non tame monodromy extension. That is, in some sense, a counter part of \cite{Br-Kr} and \cite{LRS} where an upper bound for the conductor exponent is given and it is shown that this bound is actually achieved.

In section \ref{sec:3}, one restricts to the case $p=2$ and genus $2$. In this situation there are three possible types of geometry for the stable reduction. In each case, one gives a family of curves with this degeneration type such that the wild monodromy is maximal. This has applications to the inverse Galois problem. For example, we have the following :

\begin{prop}
Let $K=\Q_2^{\rm ur}(2^{1/15})$ and $C_0/K$ the smooth, projective, geometrically integral curve given by ${Y^2=1+2^{3/5}X^2+X^3+2^{2/5}X^4+X^5}$. The irreducible components of its stable reduction $\mathcal{C}_k/k$ are elliptic curves. The monodromy extension $\M/K$ of $C_0/K$ is the decomposition field of an explicitly given polynomial. The curve $C_0/K$ has maximal wild monodromy and $G := \Gal(\M/K) \simeq Q_8 \times Q_8$. Moreover, we have
\[
 G_i \simeq  \left\{ 
				\begin{array}{ll}
				  Q_8\times  Q_8,& -1 \leq i \leq 1,\\
                 {\rm Z}( Q_8) \times  Q_8,& 2 \leq i \leq 3,  \\
                 \lbrace 1 \rbrace \times Q_8& 4 \leq i \leq 31,\\
				 \lbrace 1 \rbrace \times {\rm Z}(Q_8), & 32 \leq  i \leq 543,\\
				 \lbrace 1 \rbrace \times \lbrace 1 \rbrace, & 544 \leq  i.
                \end{array}
           \right.
\]
and $\sw(\Jac(C_0)/K)=45$.
\end{prop}
Some of the results that we give here were already available in a previous preprint of C. Lehr and  M. Matignon (see \cite{LM3}) , results about the arithmetic of the monodromy extensions, ramification and conductors are new.

\section{Background.}
\label{sec:1}
\noindent \textbf{Notations.} Let $(R,v)$ be a complete discrete valuation ring (DVR) of mixed characteristic $(0,p)$ with fraction field $K$ and algebraically closed residue field $k$. We denote by $\pi_K$ a uniformizer of $R$ and assume that $K$ contains a primitive $p$-th root of unity $\zeta_p$. Let $\lambda:=\zeta_p-1$. If $\L/K$ is an algebraic extension, we will denote by $\pi_{\L}$(resp. $v_{\L}$, resp. $\L^{\circ}$) a uniformizer for $\L$ (resp. the prolongation of $v$ to  $\L$ such that $v_{\L}(\pi_{\L})=1$, resp. the ring of integers of $\L$). If there is no possible confusion we note $v$ for the prolongation of $v$ to an algebraic closure $K^{\rm  alg}$ of $K$.\\

\addtocounter{mysub}{1}
\arabic{mysub}. \textit{Stable reduction of curves.} \label{mysub:reduction} The first result is due to Deligne and Mumford (see for example \cite{Liu} for a presentation following Artin and Winters).
\begin{thm}[Stable reduction theorem]\label{stableth}
Let $C/K$ be a smooth, projective, geometrically connected curve over $K$ of genus $g(C) \geq 2$. There exists a unique finite Galois extension $\M/K$ minimal for the inclusion relation such  that $C_{\M}/\M$ has stable reduction. The stable model $\mathcal{C}$ of $C_{\M}/\M$ over $\M^{\circ}$ is unique up to isomorphism. One has a canonical injective morphism :
\begin{equation}\label{injcano}
\Gal(\M/K) \overset{i}{\hookrightarrow} \Aut_k(\mathcal{C}_k).
\end{equation}
\end{thm}

\noindent \textbf{Remarks :}\begin{enumerate}
\item Let's explain the action of $\Gal(K^{\rm alg}/K)$ on $\mathcal{C}_k/k$. The group $\Gal(K^{\rm alg}/K)$ acts on $C_{\M}:=C\times \M$ on the right. By unicity of the stable model, this action extends to $\mathcal{C}$ :

\begin{center}
\begin{tikzpicture}[scale=1]
\node (CM1) at (-1,1) {$\mathcal{C}$};
\draw[->] (-0.5,1)--(0.5,1);
\node (CM2) at (1,1) {$\mathcal{C}$};
\node (sigma1) at (-0.05,1.2) {$\sigma$};
\draw[->] (-1,0.7)--(-1,-0.5);
\draw[->] (1,0.7)--(1,-0.5);
\node (R1) at (-1,-1) {$\M^{\circ}$};
\draw[->] (-0.5,-1)--(0.5,-1);
\node (R2) at (1,-1) {$\M^{\circ}$};
\node (sigma2) at (-0.05,-0.8) {$\sigma$};
\end{tikzpicture}
\end{center}

Since $k=k^{\rm alg}$ one gets $\sigma \times k = \Id_k$, whence the announced action. The last assertion of the theorem characterizes the elements of $\Gal(K^{{\rm alg}}/\M)$ as the elements of $\Gal(K^{{\rm  alg}}/K)$ that  trivially act on $\mathcal{C}_k/k$.

\item  If $p > 2g(C)+1$, then $C/K$ has stable reduction over a tamely ramified extension of $K$. We will study examples of covers with $p \leq 2g(C)+1$.

\item Our results will cover the elliptic case. Let $E/K$ be an elliptic curve with additive reduction. If its modular invariant is integral, then there exists a smallest extension $\M$ of $K$ over which $E/K$ has good reduction. Else $E/K$ obtains split multiplicative reduction over a unique quadratic extension of $K$ ( see \cite{Kr}).
\end{enumerate}

\begin{defi}
The extension $\M/K$ is the \textsl{monodromy extension of $C/K$}. We call $\Gal(\M/K)$ the \textsl{monodromy group of $C/K$}. It has a unique $p$-Sylow subgroup $\Gal(\M/K)_1$ called the \textsl{wild monodromy group}. The extension $\M/\M^{\Gal(\M/K)_1}$ is the \textsl{wild monodromy extension}.
\end{defi}

From now on we consider smooth, projective, geometrically integral curves $C/K$ of genus $g(C) \geq 2$ birationally given by ${Y^p=f(X):=\prod_{i=0}^t (X-x_i)^{n_i}}$ with  $(p,\sum_{i=0}^tn_i)=1$, $(p,n_i)=1$ and $\forall \; 0 \leq i \leq t, x_i \in R^{\times}$. Moreover, we assume that $\forall i \neq j, \; v(x_i-x_j)=0$, that is to say, the branch locus $B= \lbrace x_0, \dots , x_t, \infty \rbrace$ of the cover has \textsl{equidistant geometry}. We denote by $Ram$ the ramification locus of the cover.\\

\noindent \textbf{Remark :} We only ask $p$-cyclic covers to satisfy Raynaud's theorem 1' \cite{Ra} condition, that is the branch locus is $K$-rational with equidistant geometry. This has consequences on the image of \eqref{injcano}.

\begin{prop}\label{propD0}
Let $\mathcal{T}={\rm Proj}( \M^{\circ}[X_0,X_1])$ with $X=X_0/X_1$. The normalization $\mathcal{Y}$ of $\mathcal{T}$ in $K(C_{\M})$ admits a blowing-up $\tilde{\mathcal{Y}}$ which is a semi-stable model of $C_{\M}/\M$. The dual graph of $\tilde{\mathcal{Y}}_k/k$ is a tree and the points in $Ram$ specialize in a unique irreducible component $D_0 \simeq \Proj_k$ of $\tilde{\mathcal{Y}}_k/k$. There exists a contraction morphism $h : \tilde{\mathcal{Y}} \to \mathcal{C}$, where $\mathcal{C}$ is the stable model of $C_{\M}/\M$ and 
\begin{equation}\label{inj}
\Gal(\M/K) \hookrightarrow \Aut_k(\mathcal{C}_k)^{\#},
\end{equation}
where $\Aut_k(\mathcal{C}_k)^{\#}$ is the subgroup of $\Aut_k(\mathcal{C}_k)$ of elements inducing the identity on $h(D_0)$.
\end{prop}

\begin{proof}
Let $f(X)=(X-x_0)^{n_0}S(X)$ and $an_0+bp=1$. Then above $\mathcal{T} \backslash B =\Spec A$ (resp. $\mathcal{T} \backslash \lbrace B  \backslash x_0 \rbrace = \Spec A_0$), the equation of $\mathcal{Y}$ is  :
\begin{align*}
A[Y]/(Y^p-f(X)) \; \; ({\rm resp.} \; A_0[Y]/(Y^p-(X-x_0)S(X)^a)),
\end{align*}
(using \cite{Liu} 4.1.18). Since $v(S(x_0))=0$, the ramification locus $Ram$ specialize in a unique component $D_0$ of $\tilde{\mathcal{Y}}_k$. Using \cite{Ra} theorem 2, one sees that $\tilde{\mathcal{Y}}_k$ is a tree. It implies that there exists a contraction morphism $h : \tilde{\mathcal{Y}} \to \mathcal{X}$ of the components of $\tilde{\mathcal{Y}}_k$ isomorphic to $\Proj_k$ meeting $\tilde{\mathcal{Y}}_k$ in at most $2$ points (\cite{Liu} 7.5.4 and 8.3.36). The scheme $\mathcal{X}$ is seen to be stable (\cite{Liu} 10.3.31), so $\mathcal{X} \simeq \mathcal{C}$.

The component $D_0$ is smooth of genus $0$ (being birational to a curve with function field a purely inseparable extension of $K(\Proj_k)$) so $D_0 \simeq \Proj_k$. Then,  $B$ having $K$-rational equidistant geometry with $|B| \geq 3$, any element of $\Gal(\M/K)$ induces the identity on $D_0$, giving \eqref{inj}.
\end{proof}

\noindent \textbf{Remark :} The component $D_0$ is the so called \textsl{original component}.

\begin{defi}
If \eqref{inj} is surjective, we say that $C$ has \textsl{maximal monodromy}. If $v_p(|\Gal(\M/K)|)=v_p(|\Aut_k(\mathcal{C}_k)^{\#}|)$, we say that $C$ has \textsl{maximal wild monodromy}.
\end{defi}

\begin{defi}
The valuation on $K(X)$ corresponding to the discrete valuation ring $R[X]_{(\pi_K)}$ is called the \textsl{Gauss valuation $v_X$} with respect to $X$. We then have
\begin{align*}
v_X \left( \sum_{i=0}^ma_iX^i \right)= \min \lbrace v(a_i), \; 0 \leq i \leq m \rbrace.
\end{align*}
Note that a change of variables $T=\frac{X-y}{\rho}$ for $y, \rho \in R$ induces a Gauss valuation $v_T$. These valuations are exactly those that come from the local rings at generic points of components in the semi-stables models of $\Proj_K$.
\end{defi}

\addtocounter{mysub}{1}
\arabic{mysub}. \textit{Galois extensions of complete DVRs.} Let $\L/K$ be a finite Galois extension with group $G$. Then $G$ is endowed with a \textsl{lower ramification filtration} $(G_i)_{i \geq -1}$ where $G_i$ is the \textsl{$i$-th lower ramification group} defined by $G_i := \lbrace \sigma \in G \; | \; v_{\L}(\sigma(\pi_{\L})-\pi_{\L}) \geq i+1 \rbrace$. The integers $i$ such that $G_i \neq G_{i+1}$ are called \textsl{lower breaks}. For $\sigma \in G- \lbrace 1 \rbrace$, let $i_G(\sigma):= v_{\L}(\sigma(\pi_{\L})-\pi_{\L})$. The group $G$ is also endowed with a \textsl{higher ramification filtration} $(G^i)_{i \geq -1}$ which can be computed from the $G_i$'s by means of the \textsl{Herbrand's function} $\varphi_{\L/K}$. The real numbers $t$ such that $ \forall \epsilon >0, \; G^{t+\epsilon} \neq G^t$  are called \textsl{higher breaks}. We will use the following lemma (see for example \cite{Hyo}).
\begin{lem} \label{hyodo}
Let $\L/K$ defined by $X^p=1+w\pi_K^{s}$ with $ {0 < s < \frac{p}{p-1}v_K(p)}$, $(s,p)=1$ and $w \in R^{\times}$. The different ideal $\mathcal{D}_{\L/K}$ satisfies :
$$ v_K(\mathcal{D}_{\L/K})=v_K(p)+\frac{p-1}{p}(1-s).$$
\end{lem}

\addtocounter{mysub}{1}
\arabic{mysub}. \textit{Extra-special $p$-groups.} The Galois groups and automorphism groups that we will have to consider are $p$-groups with peculiar group theoretic properties (see for example \cite{Su} for an account on extra-special $p$-groups). We will denote by  ${\rm Z}(G)$ (resp. ${\rm D}(G)$, $\Phi(G)$) the center (resp. the derived subgroup,  the Frattini subgroup) of $G$. If $G$ is a $p$-group, one has  $\Phi(G)={\rm D}(G)G^p$.

\begin{defi}
An \textsl{extra-special $p$-group} is a non abelian $p$-group $G$ such that $ {\rm D}(G)={\rm Z}(G)=\Phi(G)$ has order $p$.
\end{defi}

\begin{prop}
Let $G$ be an extra-special $p$-group.
\begin{enumerate}
\item Then $|G|=p^{2n+1}$ for some $n \in \N^{\times}$.
\item One has the exact sequence 
\begin{align*}
0 \to {\rm Z}(G) \to G \to (\Z/p\Z)^{2n} \to 0.
\end{align*}
\end{enumerate}
\end{prop}

\noindent \textbf{Remark :} With the previous notations, we will encounter curves such that the $p$-Sylow subgroup of $\Aut_k(\mathcal{C}_k)^{\#}$ is an extra-special $p$-group. In this case, the above short exact sequence has a geometric description that we will make explicit later on.\\

\addtocounter{mysub}{1}
\arabic{mysub}. \textit{Torsion points on abelian varieties.} Let $A/K$ be an abelian variety over $K$ with  potential good reduction. Let $ \ell \neq p$ be a prime number, we denote by $A[\ell]$ the $\ell$-torsion subgroup of $A(K^{\rm  alg})$ and by $T_{\ell}(A)= \varprojlim A[\ell^n]$ (resp. $V_{\ell}(A)=T_{\ell}(A) \otimes\Q_{\ell}$) the Tate module (resp. $\ell$-adic Tate module) of $A$.

\indent The following result may be found in \cite{Gur} (paragraph 3). We recall it for the convenience of the reader.

\begin{lem}\label{guralnick}
Let $k=k^{\rm alg}$  be a field with ${\rm char} \; k=p \geq 0$ and $C/k$ be a projective, smooth, integral curve. Let $\ell \neq p$ be a prime number and $H$ be a finite subgroup of $\Aut_k(C)$ such that $(|H|,\ell)=1$. Then 
\[2g(C/H)= \dim_{\F_{\ell}} \Jac(C)[\ell]^H\].
\end{lem}

\indent If $\ell \geq 3$, then  $\L=K(A[\ell])$ is the minimal extension over which $A/K$ has good reduction. It is a Galois extension with group $G$ (see \cite{SerTat}). We denote by $r_G$ (resp. $1_G$) the character of the regular (resp. unit) representation of $G$. We denote by $I$ the inertia group of $K^{\rm alg}/K$. For further explanations about  conductor exponents see \cite{Ser2}, \cite{Ogg} and \cite{SerTat}.
\begin{defi}
\begin{enumerate}
\item Let 
\begin{align*}
a_G(\sigma) & := -i_G(\sigma), \; \; \; \;\sigma \neq 1,\\
a_G(1)      & := \sum_{\sigma \neq 1} i_G(\sigma),
\end{align*}
 and $\sw_G:=a_G-r_G+1_G$. Then, $a_G$ is the character of a $\Q_{\ell}[G]$-module and there exists a projective $\Z_{\ell}[G]$-module $\Sw_G$ such that $\Sw_G \otimes_{\Z_{\ell}} \Q_{\ell}$ has character $\sw_G$.
\item We still denote by $T_{\ell}(A)$ (resp. $A[\ell]$) the $\Z_{\ell}[G]$-module (resp. $\F_{\ell}[G]$-module) afforded by $G \to \Aut(T_{\ell}(A))$ (resp. $ G \to \Aut(A[\ell])$). Let 
\begin{align*}
\sw(A/K)&:=\dim_{\F_{\ell}} \Hom_G(\Sw_G,A[\ell]),\\
\epsilon(A/K)&:=\codim_{\Q_{\ell}} V_{\ell}(A)^I.
\end{align*}
The integer $f(A/K):=\epsilon(A/K)+\sw(A/K)$ is the so called \textsl{conductor exponent of $A/K$} and $\sw(A/K)$ is the \textsl{Swan conductor of $A/K$}.
\end{enumerate}
\end{defi}

\begin{prop}Let $\ell \neq p$, $ \ell \geq 3$ be a prime number.
\begin{enumerate} 
\item The integers $\sw(A/K)$ and $\epsilon(A/K)$ are independent of $\ell$.
\item One has 
\[ \sw(A/K) = \sum_{i \geq 1} \frac{|G_i|}{|G_0|} \dim_{\F_{\ell}} A[\ell]/A[\ell]^{G_i}.\]
Moreover, for $\ell$ large enough, $\epsilon(A/K)=\dim_{\F_{\ell}}A[\ell]/A[\ell]^{G_0}$.
\end{enumerate}
\end{prop}
\noindent \textbf{Remark :} It follows from the definition that $\sw(A/K)=0$ if and only if $G_1= \lbrace 1 \rbrace$. The Swan conductor is a measure of the wild ramification.\\

\section{Covers with potential good reduction.}
\label{sec:2}
We start by fixing notations that will be used throughout this section.\\

\noindent \textbf{Notations.} We denote by $\mathfrak{m}$ the maximal ideal of $ (K^{\rm alg})^{\circ}$. Let $n \in \N^{\times}$, ${q=p^n}$ and $a_n=(-1)^q(-p)^{p+p^2+\dots+q}$. We denote by $\Qpur$ the maximal unramified extension of $\Q_p$. Let $K:=\Qpur(\lambda^{1/(1+q)})$. For $c \in R$ let 
\[f_{q,c}(X)=1+cX^q+X^{1+q}. \]
One defines the \textsl{modified monodromy polynomial} by
\[ L_c(X)=X^{q^2}-a_n(c+X)f_{q,c}(X)^{q-1}. \]
Let $C_c / K $ and  $A_q / k$ be the smooth projective integral curves birationally given respectively by $Y^p=f_{q,c}(X)$ and $w^p-w=t^{1+q}$. 

\begin{thm}
The curve $C_c/K$ has potential good reduction isomorphic to $A_q/k$.
\begin{enumerate}
\item If $v(c) \geq v(\lambda^{p/(1+q)})$, then the monodromy extension of $C_c/K$ is trivial.
\item If $v(c) < v(\lambda^{p/(1+q)})$, let $y$ be a root of $L_c(X)$ in $K^{\rm  alg}$. Then $C_c$ has good reduction over $K(y,f_{q,c}(y)^{1/p})$. If $L_c(X)$ is irreducible over $K$, then $C_c/K$ has maximal wild monodromy. The monodromy extension of $C_c/K$ is $\M=K(y,f_{q,c}(y)^{1/p})$ and $G=\Gal(\M/K)$ is an extra-special $p$-group of order $pq^2$. If $c \in R$ with $v(c^{p}-c) \geq v(p)$, then $L_c(X)$ is irreducible over $K$, the lower ramification filtration of $G$ is
\[G= G_0 = G_1 \supsetneq G_2 = \dots = G_{1+q} = {\rm Z}(G) \supsetneq \lbrace 1 \rbrace . \]
Moreover, one has $f(\Jac(C_c)/K)=(2q+1)(p-1)$. If $c \in \Qpur$ then $\sw(\Jac(C_c)/\Qpur)=1$.
\end{enumerate}
\end{thm} 

\begin{proof}

\textit{1.} Assume that $v(c) \geq v(\lambda^{p/(1+q)})$. Set $\lambda^{p/(1+q)}T=X$ and ${\lambda W+1=Y}$. Then, the equation defining $C_c/K$ becomes
\[ (\lambda W+1)^p=\sum_{i=0}^p \binom{p}{i}\lambda^iW^i =1+c\lambda^{pq/(1+q)}T^q+\lambda^pT^{1+q}. \]
After simplification by $\lambda^p$ and reduction modulo $\pi_K$ this equation gives :
\begin{equation}\label{equa:1}
 w^p-w=at^q+t^{1+q}, \; a \in k.
\end{equation}
By Hurwitz formula the genus of the curve defined by \eqref{equa:1} is seen to be that of $C_c/K$ . Applying  \cite{Liu} 10.3.44, there is a component in the stable reduction birationally given by \eqref{equa:1}. The stable reduction being a tree, the curve $C_c/K$ has good reduction over $K$.\\

\noindent \textit{2.} The proof of the first part is divided into six steps. Let $y$ be a root of $L_c(X)$. \\

\noindent \textbf{Step I :} \textsl{One has $v(y)=v(a_nc)/q^2$.}

\noindent Since $y$ is a root of $L_c(X)$, one has  
\begin{equation*}
 y^{q^2}=a_n(c+y)f_{q,c}(y)^{q-1},
\end{equation*}
so $v(y) > 0$. Assume that $v(c+y) \geq v(y)$. Then, $q^2v(y) \geq  v(a_n)+v(y)$ and $v(c) \geq v(y) \geq \frac{v(a_n)}{q^2-1} = \frac{p}{q+1}v(\lambda)$, which is a contradiction. So $v(c+y) < v(y)$ thus $v(c+y)=v(c)$.\\

\noindent \textbf{Step II :} \textsl{Define $S$ and $T$ by $\lambda^{p/(1+q)}T=(X-y)=S$. Then, }
\begin{align*}
f_{q,c}(S+y) \equiv f_{q,c}(y)+y^qS+(c+y)S^q+S^{1+q} \mod \lambda^p \mathfrak{m}[T].
\end{align*}
Using the following formula for $A \in K^{\rm alg}$ with $v(A)>0$ and $B \in (K^{\rm alg})^{\circ}[T]$
\begin{equation*}
(A+B)^q \equiv (A^{q/p}+B^{q/p})^p \mod p^2 \mathfrak{m}[T], 
\end{equation*}
one computes$\mod \lambda^p \mathfrak{m}[T]$
\begin{align*}
f_{q,c}(y+S) &= 1+c(y+S)^q+(y+S)^{1+q}\\
         &\equiv 1+c(y^{q/p}+S^{q/p})^p+   (y+S)(y^{q/p}+S^{q/p})^p \\
         &\equiv f_{q,c}(y)+(y^q + \Sigma)S + (c+y)S^q + S^{1+q}+(c+y)\Sigma ,
\end{align*}
where $\Sigma = \sum_{k=1}^{p-1} \binom{p}{k}y^{kq/p}S^{(p-k)q/p}$. Using \textbf{Step I}, one checks that ${\Sigma \in \lambda^p \mathfrak{m}[T]}$.\\

\noindent \textbf{Step III :} \textsl{Let $R_1:=K[y]^{\circ}$. For all $0 \leq i \leq n$, there exist $B_i \in R_1$ and ${A_i(S) \in R_1[S]}$ such that ${\rm mod \;}\lambda^p \mathfrak{m}[T]$ one has :}
\begin{equation}\label{induction}
f_{q,c}(S+y) \equiv f_{q,c}(y)(1+SA_i(S))^p+y^qS+B_iS^{q/p^i}+S^{1+q} .
\end{equation} 
One defines the $A_i(S)$'s and the $B_i$'s by induction. For all $0 \leq i \leq n-1$, let 
\begin{align*}
B_n:= &-y^q & {\rm and} \; \;&B_i:= f_{q,c}(y) \frac{B_{i+1}^p}{(-pf_{q,c}(y))^p},\\
A_0(S):=& 0 & {\rm and}\; \; &SA_{i+1}(S):= SA_i(S)-p^{-1}f_{q,c}(y)^{-1}B_{i+1}S^{q/p^{i+1}}.
\end{align*}
One checks that for all $ 0 \leq i \leq n$ 
\begin{equation*}
B_i/f_{q,c}(y)=(-p)(-p)^{-1-p-\dots- p^{n-i}}(-y^q/f_{q,c}(y))^{p^{n-i}},
\end{equation*}
and  
\begin{equation}
v(B_i) \geq (1+\frac{1}{p}+ \dots + \frac{1}{p^{i-1}})v(p),\; \forall 1 \leq i \leq n.
\end{equation}
It follows that $\forall \; 1 \leq i \leq n$, $p^{-1}B_i \in R_1$, $\forall \; 0 \leq i \leq n$  $A_i(S) \in R_1[S]$ and  $B_0=c+y$ since $L_c(y)=0$. \\
\indent One proves this step by induction on $i$. According to \textbf{Step II}, the equation	 \eqref{induction} holds for $i=0$. Assume that \eqref{induction} is satisfied for $i$. Taking into account that
\begin{align*}
&f_{q,c}(y)(1+(-1)^p)\frac{B_{i+1}^p}{p^pf_{q,c}(y)^p}S^{q/p^i},\\
&\sum_{k=2}^{p-1}\binom{p}{k}(\frac{B_{i+1}}{pf_{q,c}(y)}S^{q/p^{i+1}})^k(1+SA_{i+1}(S))^{p-k},\\
&B_{i+1}S^{q/p^{i+1}}\sum_{k=1}^{p-1}\binom{p-1}{k}S^kA_{i+1}(S)^k ,
\end{align*}
are in $ \lambda^p \mathfrak{m}[T]$, one gets \eqref{induction} for $i+1$.\\

\noindent \textbf{Step IV :} \textsl{The curve $C_c/K$ has good reduction over $K(y,f_{q,c}(y)^{1/p})$.} 

\noindent Applying \textbf{Step III} for $i=n$, one gets 
\begin{equation*}
f_{q,c}(y+S) \equiv f_{q,c}(y)(1+SA_n(S))^p+S^{1+q} \mod \lambda^p \mathfrak{m}[T],
\end{equation*}
then the change of variables in $K(y,f_{q,c}(y)^{1/p})$
\begin{align*}
X=\lambda^{p/(1+q)}T+y=S+y \; \;{\rm and} \; \; \frac{Y}{f_{q,c}(y)^{1/p}}=\lambda W+1+SA_n(S),
\end{align*}
induces in reduction $w^p-w=t^{1+q}$ with genus $g(C_c)$. So \cite{Liu} 10.3.44 implies that the above change of variables gives the stable model.\\

\noindent \textbf{Step V :} \textsl{For any distinct roots $y_i$, $y_j$ of $L_c(X)$, $v(y_i-y_j) = v(\lambda^{p/(1+q)})$.}

\noindent The changes of variables $T=(X-y_i)/\lambda^{p/(1+q)}$ and $T=(X-y_j)/\lambda^{p/(1+q)}$ induce equivalent Gauss valuations of $K(C_c)$ else applying \cite{Liu} 10.3.44 would contradict the uniqueness of the stable model. In particular ${v(y_i-y_j)} \geq v(\lambda^{p/(1+q)})$.\\
\indent Using \textbf{Step I}, one checks that $v(q^2y^{q^2-1}) > v(a_n)$ and $v(f_{q,c}'(y)) > 0$, so :
\begin{equation*}
v(L_c'(y))=v(a_n)=(q^2-1)v(\lambda^{p/(1+q)}).
\end{equation*}
Taking into account that $L_c'(y_i)=\prod_{j \neq i}(y_i-y_j)$ and $\deg L_c(X)=q^2$, one obtains $v(y_i-y_j) = v(\lambda^{p/(1+q)})$.\\

\noindent \textbf{Step VI :} \textsl{If $L_c(X)$ is irreducible over $K$, then $K(y,f_{q,c}(y)^{1/p})$ is the monodromy extension $\M$ of $C_c/K$ and $G:=\Gal(\M/K)$ is an extra-special $p$-group of order $pq^2$.}

\noindent Let $(y_i)_{i=1,\dots,q^2}$ be the roots of $L_c(X)$, $\L:=K(y_1, \dots, y_{q^2})$ and  $\M/K$ be the monodromy extension of $C_c/K$. Any  $\tau \in \Gal(\L/K) -\lbrace 1 \rbrace$ is such that $\tau(y_i)=y_j$ for some $i \neq j$. Thus, the change of variables 
\begin{align*}
X=\lambda^{p/(1+q)}T+y_i \; \;{\rm and} \; \; \frac{Y}{f_{q,c}(y_i)^{1/p}}=\lambda W+1+SA_n(S),
\end{align*}
 induces the stable model and $\tau$ acts on it by :
\begin{equation*}
\tau(T)=\frac{X-y_j}{\lambda^{p/(1+q)}}, \; \; \;  {\rm hence} \; \; \;  T-\tau(T)=\frac{y_j-y_i}{\lambda^{p/(1+q)}}.
\end{equation*}
According to \textbf{Step V}, $\tau$ acts non-trivially on the stable reduction. It follows that $\L \subseteq \M$. Indeed if $\Gal(K^{\rm alg}/\M) \nsubseteq \Gal(K^{\rm alg}/\L)$ it would exist $ {\sigma \in \Gal(K^{\rm alg}/\M)}$ inducing $\bar{\sigma} \neq \Id \in \Gal(\L/K)$, which would contradict the characterization of $\Gal(K^{\rm alg}/\M)$ (see remark after Theorem \ref{stableth}) . \\
\indent According to \cite{LM}, the $p$-Sylow subgroup $\Aut_k(\mathcal{C}_k)_1^{\#}$ of $\Aut_k(\mathcal{C}_k)^{\#}$ is an extra-special $p$-group of order $pq^2$. Moreover, one has  :
\begin{align*}
0 \to {\rm Z}(\Aut_k(\mathcal{C}_k)_1^{\#})\to \Aut_k(\mathcal{C}_k)_1^{\#} \to (\Z/p\Z)^{2n} \to 0,
\end{align*}
where $(\Z/p\Z)^{2n}$ is identified with the group of translations $t \mapsto t+a$ extending to elements of $\Aut_k(\mathcal{C}_k)_1^{\#}$. Therefore we have morphisms 
\begin{equation*}
 \Gal(\M/K) \overset{i}{\hookrightarrow} \Aut_k(\mathcal{C}_k)_1^{\#} \overset{\varphi}{\to} \Aut_k(\mathcal{C}_k)_1^{\#} /{\rm Z}(\Aut_k(\mathcal{C}_k)_1^{\#}).
\end{equation*}
The composition is seen to be surjective since the image contains the $q^2$ translations 
$t \mapsto t+ \overline{(y_i-y_1)/\lambda^{p/(1+q)}}$. Consequently, $i(\Gal(\M/K))$ is a subgroup of $\Aut_k(\mathcal{C}_k)_1^{\#}$ of index at most $p$. So it contains $\Phi(\Aut_k(\mathcal{C}_k)_1^{\#})={\rm Z}(\Aut_k(\mathcal{C}_k)_1^{\#})=\Ker \varphi$. It implies that $i$ is an isomorphism. Thus $\left[\M:K\right]=pq^2$. By \textbf{Step IV}, one has $\M \subseteq K(y,f_{q,c}(y)^{1/p})$, hence $\M = K(y,f_{q,c}(y)^{1/p})$.\\

\indent We show, for later use, that $K(y_1)/K$ is Galois and that $\Gal(\M/K(y_1)) ={\rm Z}(G)$. Indeed, $\M/K(y_1)$ is $p$-cyclic and generated by $\sigma$ defined by :
\[ \sigma(y_1)=y_1 \;  {\rm and} \; \sigma(f_{q,c}(y_1)^{1/p})=\zeta_p^{-1}f_{q,c}(y_1)^{1/p}. \]
According to \textbf{Step IV}, $\sigma$ acts on the stable model by :
\begin{equation*}
\sigma(S)=S, \; \; \;  \sigma(\frac{Y}{f_{q,c}(y_1)^{1/p}})=\frac{Y}{\zeta_p^{-1}f_{q,c}(y_1)^{1/p}}=\lambda \sigma(W)+1+SA_n(S).
\end{equation*}
Hence 
\begin{align*} &\frac{\lambda W+1+SA_n(S)}{\zeta_p^{-1}}=\lambda \sigma(W)+1+SA_n(S),\\
{\rm thus,} \; \;  \; \; \;  & \sigma(W)=\zeta_p W+1+SA_n(S).
\end{align*}
It follows that, in reduction, $\sigma$ induces the Artin-Schreier morphism that generates ${\rm Z}(\Aut_k(\mathcal{C}_k)_1^{\#})$. It implies that $K(y_1)/K$ is Galois, $\Gal(\M/K(y_1)) ={\rm Z}(G)$ and $\Gal(K(y_1)/K)  \simeq (\Z/p\Z)^{2n}$.\\

\noindent We now prove the statements concerning the arithmetic of $\M/K$. We assume that $c \in R$ with $v(c^p-c) \geq v(p)$ and we split the proof into $5$ steps. Let $y$ be a root of $L_c(X)$ and $b_n:=(-1)(-p)^{1+p+\dots+p^{n-1}}$. Note that  $b_n^p=a_n$ and $L:=K(y_1, \dots, y_{q^2})=K(y_1)$. We note that $v(c^p-c) \geq v(p)$, so $v(\lambda^{p/(1+q)}) > v(c)$ implies $v(c)=0$.\\

\noindent \textbf{Step A :} \textsl{The polynomial $L_c(X)$ is irreducible over $K$.}

\noindent One computes 
\begin{align*}
(y^{q^2/p}-cb_n)^p&=y^{q^2}+(-c)^pa_n+\Sigma \\
                  &=a_n(1+y^q(c+y))^{q-1}(c+y)+(-c)^pa_n+\Sigma\\
                  &=a_n \sum_{k=0}^{q-1}\binom{q-1}{k}y^{kq}(c+y)^{1+k}+(-c)^pa_n+ \Sigma\\
                  &=a_ny+a_n(c+(-c)^p)+a_n\Sigma'+\Sigma,
\end{align*}
where $\Sigma :=\sum_{k=1}^{p-1}\binom{p}{k}y^{kq^2/p}(-cb_n)^{p-k}$ and $\Sigma':=\sum_{k=1}^{q-1}\binom{q-1}{k}y^{kq}(c+y)^{1+k}$. Using \textbf{Step I} one checks that $v(\Sigma) > v(a_ny)$ and $v(\Sigma')\geq v(y^q) > v(y)$. Since $v(c^p-c) \geq v(p) > v(y)$, one gets :
\begin{equation*}
v(y^{q^2/p}-cb_n)=\frac{v(a_ny)}{p},
\end{equation*}
and $t:=p^{q^2}(y^{q^2/p}-cb_n)^{-(p-1)(q+1)} \in \L$ has valuation $v_{\L}(p)/q^2= [\L:\Qpur]/q^2$. So $q^2$ divides $[\L:K]$. It implies that $L_c(X)$ is irreducible over $K$.\\

\noindent \textbf{Step B :} \textsl{Reduction step.}

\noindent The last non-trivial group $G_{i_0}$ of the lower ramification filtration $(G_i)_{i \geq 0}$ of $G:=\Gal(\M/K)$  is a subgroup of ${\rm Z}(G)$ (\cite{Ser} IV \S 2 Corollary 2 of Proposition 9) and as ${\rm Z}(G) \simeq \Z/p\Z$, it follows that $G_{i_0} ={\rm Z}(G)$.\\
\indent According to \textbf{Step VI} the group $H:=\Gal(\M/ \L)$ is ${\rm Z}(G)$. Consequently, the filtration $(G_i)_{i \geq 0}$ can be deduced from that of $\M/\L$ and $\L/K$ (see \cite{Ser} IV \S 2 Proposition 2 and Corollary of Proposition 3).\\

\noindent \textbf{Step C :} \textsl{The ramification filtration of $L/K$ is :}
\begin{equation*}
(G/H)_0=(G/H)_1 \supsetneq (G/H)_2 = \lbrace 1 \rbrace.
\end{equation*}

\noindent Since $K/\Qpur$ is tamely ramified of degree $(p-1)(q+1)$, one has $K=\Qpur(\pi_K)$ with $\pi_K^{(p-1)(q+1)}=p$ for some uniformizer $\pi_K$ of $K$. In particular,
\begin{equation*}
z:=\frac{\pi_K^{q^2}}{y^{q^2/p}-cb_n},
\end{equation*}
is such that $t=z^{(p-1)(q+1)}$. Then, following the proof of \textbf{Step A}, $z$ is a uniformizer of $\L$. Let $y$ and $y'$ be two distinct roots of $L_c(X)$. Let ${\sigma \in \Gal(\L/K)}$ such that $\sigma(y)=y'$. Then 
\begin{align*}
\sigma(z)-z &=\frac{\pi_K^{q^2}}{y'^{q^2/p}-cb_n}-\frac{\pi_K^{q^2}}{y^{q^2/p}-cb_n}\\
           &=\pi_K^{q^2}\frac{y^{q^2/p}-y'^{q^2/p}}{(y^{q^2/p}-cb_n)(y'^{q^2/p}-cb_n)},
\end{align*}
so $v(\sigma(z)-z)=2v(z)-q^2v(\pi_K)+v(y'^{q^2/p}-y^{q^2/p})$. It follows from :
\begin{equation*}
(y-y')^{q^2/p}=y^{q^2/p}+(-y')^{q^2/p}+\sum_{k=1}^{\frac{q^2}{p}-1}\binom{q^2/p}{k}y^k(-y')^{\frac{q^2}{p}-k},
\end{equation*}
and $v(y)=v(y')$, $v(p)+\frac{q^2}{p}v(y) > \frac{q^2}{p}v(y-y')$ (use \textbf{Step I} and \textbf{Step V}) that $v(y^{q^2/p}-y'^{q^2/p})= \frac{q^2}{p}v(y-y')=q^2v(\pi_K)$. Hence $v(\sigma(z)-z)=2v(z)$. This means that $(G/H)_2=\lbrace 1 \rbrace$.\\

\noindent \textbf{Step D :} \textsl{Let $s:=(q+1)(pq^2-1)$. There exist $u \in L$, $ r \in \pi_{\L}^s \mathfrak{m}$ such that 
\begin{equation*}
f_{q,c}(y)u^p = 1+ py^{q/p}(\frac{y^{q^2/p}}{b_n}-c) +r,
\end{equation*}
 and $v_{\L}( py^{q/p}(\frac{y^{q^2/p}}{b_n}-c))=s$.}

\noindent To prove the second statement, we note that :
\begin{align*}
(\frac{y^{q^2/p}}{b_n})^p &=f_{q,c}(y)^{q-1}(c+y) =\sum_{k=0}^{q-1}\binom{q-1}{k}y^{kq}(c+y)^{1+k}=c+y+\Sigma,
\end{align*}
with $\Sigma:=\sum_{k=1}^{q-1}\binom{q-1}{k}y^{kq}(c+y)^{1+k} $. We set $h:=\frac{y^{q^2/p}}{b_n}-c$ and compute  :
\begin{align*}
h^p&=(\frac{y^{q^2/p}}{b_n})^p+(-c)^p+\sum_{k=1}^{p-1}\binom{p}{k}(\frac{y^{q^2/p}}{b_n})^k(-c)^{p-k}\\
                            &=c+(-c)^p+y+\Sigma+\sum_{k=1}^{p-1}\binom{p}{k}(\frac{y^{q^2/p}}{b_n})^k(-c)^{p-k}.
\end{align*}
Since $v(c^p-c) \geq v(p) > v(y)$, $v(\Sigma) >v(y)$ and  $v(\frac{y^{q^2/p}}{b_n}) \geq 0$, one gets $v_{\L}(h)=v_{\L}(y)/p=q^2-1$ and $v_{\L}( py^{q/p}h)=s$.

\indent For the first claim, if $n \geq 2$ we choose :
\begin{equation*}
u:=1-cy^{q/p}+\sum_{k=0}^{n-2}\frac{y^{(1+q)p^k}}{(-p)^{1+p+\dots+p^k}}=1+w.
\end{equation*}
Then, $f_{q,c}(y)u^p-1 = 1+cy^q+y^{1+q}+\Sigma_1 +cy^q\Sigma_1 +y^{1+q}\Sigma_1 -1$ with :
\begin{align*}
\Sigma_1 &:=\sum_{k=1}^{p-1}\binom{p}{k}w^k+w^p=pw+\sum_{k=2}^{p-1}\binom{p}{k}w^k+w^p=pw+\Sigma' +w^p\\
      &=p\left[-cy^{q/p}-\frac{y^{1+q}}{p}+\sum_{k=1}^{n-2}\frac{y^{(1+q)p^k}}{(-p)^{1+p+\dots+p^k}}\right]+\Sigma' +w^p,
\end{align*}
and $\Sigma' :=\sum_{k=2}^{p-1}\binom{p}{k}w^k$, where sums are eventually empty if $n = 2$ or $p=2$. So :
\begin{align*}
f_{q,c}(y)u^p-1 &= cy^q-pcy^{q/p}+\sum_{k=1}^{n-2}\frac{py^{(1+q)p^k}}{(-p)^{1+p+\dots+p^k}}+\Sigma'+ w^p  \\
&+cpy^qw +cy^q\Sigma' +cy^qw^p+y^{1+q}pw+y^{1+q}\Sigma'+y^{1+q}w^p.
\end{align*}
Computation shows that $v(w)=v(y^{q/p})$ and we deduce that :
\begin{align*} w^p &= (-c)^py^q+\sum_{k=0}^{n-2}\frac{y^{(1+q)p^{1+k}}}{(-p)^{p+\dots+p^{1+k}}} \mod \pi_{\L}^s\mathfrak{m}.
\end{align*}
One checks that $v_{\L}(y^qp)>s$, $v_{\L}(y^qw^p)>s$ and $v_{\L}(\Sigma')>s$. Hence : 
\begin{align*}
f_{q,c}(y)u^p-1 &= (c+(-c)^p)y^q-pcy^{q/p}+\sum_{k=1}^{n-2}\frac{py^{(1+q)p^k}}{(-p)^{1+\dots+p^k}}+\sum_{k=0}^{n-2}\frac{y^{(1+q)p^{1+k}}}{(-p)^{p+\dots+p^{1+k}}}\\
                &= -pcy^{q/p}+\frac{y^{q/p(1+q)}}{(-p)^{p+\dots+q/p}} = py^{q/p}(\frac{y^{q^2/p}}{b_n} -c) \mod \pi_{\L}^s\mathfrak{m}.
 \end{align*}
\noindent If $n=1$, we choose $u:=1-cy$ and check that the statement is still true.\\

\noindent \textbf{Step E :} \textsl{Computation of conductors.}

\noindent From \textbf{Step D}, one deduces that the extension $\M/\L$ is defined by $X^p=1+phy^{q/p}+r$ with $r \in \pi_{\L}^s\mathfrak{m}$. From lemma \ref{hyodo}, one gets that $v_{\M}(\mathcal{D}_{\M/\L})=(p-1)(q+2)$. Hence  
\[
\Z/p\Z \simeq H_0=H_1=\dots =H_{1+q} \supsetneq \lbrace  1 \rbrace,
\]
and according to \textbf{Step B} and \textbf{Step C} one has :
\begin{equation*}
G=G_0=G_1 \supsetneq {\rm Z}(G)=G_2= \dots =G_{1+q} \supsetneq \lbrace 1 \rbrace .
\end{equation*}

\noindent Let $\ell \neq p$ be a prime number. Since the $G$-modules $\Jac(C)[\ell]$ and $\Jac(\mathcal{C}_k)[\ell]$ are isomorphic (see \cite{SerTat} paragraph 2) one has  that for $i \geq 0$ :
\[ \dim_{\F_{\ell}} \Jac(C)[\ell]^{G_i}= \dim_{\F_{\ell}}\Jac(\mathcal{C}_k)[\ell]^{G_i}.
\]
Moreover, for $0 \leq i \leq 1+q$ one has $\Jac(\mathcal{C}_k)[\ell]^{G_i} \subseteq \Jac(\mathcal{C}_k)[\ell]^{{\rm Z}(G)}$. Then, from $\mathcal{C}_k/{\rm Z}(G) \simeq \Proj_k$ (see end of \textbf{StepVI}) and lemma \ref{guralnick}, it follows that  for $0 \leq i \leq 1+q$, $\dim_{\F_{\ell}}\Jac(\mathcal{C}_k)[\ell]^{G_i} = 0$. Since $g(C)=q(p-1)/2$, one gets $f(\Jac(C)/K)=(2q+1)(p-1)$. Moreover, if $c \in \Qpur$, an easy computation shows that $\sw(\Jac(C)/\Qpur)=1$.
\end{proof}

\noindent \textbf{Remark :} If $c \in R$ with $v(c)=(a/b)v(p) < v(\lambda^{p/(1+q)})$ and $a$ and $b$ both prime to $p$, then $L_c(X)$ is irreducible over $K$. Indeed, the expression of the valuation of any root $y$ of $L_c(X)$ shows that the ramification index of $K(y)/K$ is $q^2$.

\section{Monodromy of genus $2$ hyperelliptic curves}
\label{sec:3}

We restrict to the case $p=2$ and $\deg f(X)=5$ of the introduction. In this situation, there are three types of geometry for the stable reduction (see Figure 1). For each type of degeneration, we will give an example of cover $C/K$ with maximal wild monodromy and birationally given by $Y^2=f(X)=1+b_2X^2+b_3X^3+b_4X^4+X^5 \in R[X]$ over some $R$. Define $\mathcal{X}$ to be the $R$-model of $C/K$ given by $Y^2=f(X)$ and let's describe each degeneration type.\\
\indent The Jacobian criterion shows that $\mathcal{X}_k/k$ has two singularities if and only if $\overline{b_3} \neq 0$. Type I occurs when $\mathcal{X}_k/k$ has two singularities and by blowing-up $\mathcal{X}$ at these points one obtains two elliptic curves. Type II (resp. type III) occurs in the one singularity case when there are two (resp. one) irreducible components of non $0$ genus in the stable reduction.
The elliptic curves $E/k$ that we will encounter are birationally given by $w^2-w=t^3$. They are such that $\Aut_k(E) \simeq \Sl_2(\F_3)$ has a unique $2$-Sylow subgroup isomorphic to $Q_8$. We denote by $D_0$ the \textsl{original component} defined in Proposition \ref{propD0}.
\newline

\begin{tikzpicture}[scale=4.1]

\draw (-1.2,-1)--(-0.5,-1);
\draw (-1.2,-1)--(-1.1,-0.3);
\draw (-0.55,-0.9)--(-0.45,-0.3);
\node (D0I) at (-0.4,-0.35) {\tiny $D_0$};
\draw (-0.45,-0.7) .. controls (-0.55,-0.80) and (-1,-0.6) .. (-1,-0.4);
\node (T1) at (-0.9,-1.2) {\small{Type $III$}};

\draw (-0.1,-1)--(0.6,-1);
\draw (-0.1,-1)--(0,-0.3);
\draw (0.55,-0.9)--(0.65,-0.3);
\node (D0II) at (0.7,-0.35) {\tiny $D_0$};
\draw (0.65,-0.5) .. controls (0.45,-0.6) and (0.1,-0.5) .. (0.1,-0.3);
\draw (0.60,-0.8) .. controls (0.4,-0.9) and (0.11,-0.8) .. (0.1,-0.6);
\node (T2) at (0.2,-1.2) {\small{Type $I$}};

\draw (1,-1)--(1.7,-1);
\draw (1,-1)--(1.1,-0.3);
\draw (1.65,-0.9)--(1.75,-0.3);
\node (D0III) at (1.8,-0.35) {\tiny $D_0$};
\draw (1.1,-0.65)--(1.75,-0.75);
\draw (1.3,-0.2)..controls (1.2,-.25) and  (1.2,-0.75) .. (1.3,-0.8);
\draw (1.5,-0.2)..controls (1.4,-.25) and  (1.4,-0.75) .. (1.5,-0.8);
\node (T3) at (1.3,-1.2) {\small{Type $II$}};

\node (fig) at (0.2,-1.4) {Figure 1};

\end{tikzpicture}

Magma codes used in this section are available on the authors webpages \cite{Web}.\\

\noindent \textbf{Notations.} For $f(X) \in R[X]$, let 
\begin{equation*}
f(X+x)=s_0(x)+s_1(x)X+s_2(x)X^2+s_3(x)X^3+s_4(x)X^4+X^5,
\end{equation*}
be the Taylor expansion of $f$ and define
\begin{equation*}
T_f(Y):=s_1(Y)^2-4s_0(Y)s_2(Y).
\end{equation*}

\textit{Degeneration type III :} This is the case of potential good reduction. For example, using notations of the previous section, let $K:=\Q_2^{\rm ur}((-2)^{1/5})$ and $C_1/K$ be the smooth, projective, geometrically integral curve birationally given by $Y^2=1+X^4+X^5=f_{4,1}(X)$. Then, $C_1/K$ has potential good reduction with maximal wild monodromy $\M/K$, the group $\Gal(\M/K)$ is an extra-special $2$-group of order $2^5$, $f(\Jac(C_1)/K)=9$ and $\sw(\Jac(C_1)/\Q_2^{\rm ur})=1$. \\

\textit{Degeneration type I :}
\begin{prop}\label{proptypeii} Let $\rho:= 2^{2/3}$, $b_2,b_3,b_4 \in \Q_2^{\rm alg}$, $K:=\Q_2^{\rm ur}(b_2,b_3,b_4)$ and $C/K$ be the smooth, projective, geometrically integral curve birationally given by
\begin{equation*}
Y^2=f(X)=1+b_2X^2+b_3X^3+b_4X^4+X^5,
\end{equation*}
with $v(b_i) \geq 0$ and $v(b_3)=0$. Assume that $1+b_3b_2+b_3^2b_4 \not \equiv 0 \mod \pi_K$. Then $C$ has stable reduction of type I and 
\begin{equation*}
T_f(Y)=T_{1,f}(Y)T_{2,f}(Y) \; \; {\rm  with} \; T_{i,f}(Y) \in K[Y],
\end{equation*}
such that $\overline{T_{1,f}}(Y)=Y^4 \in k[Y]$ and $\overline{T_{2,f}}(Y)=Y^4+\overline{b_3}^2 \in k[Y]$. If $T_{1,f}(Y)$ and $T_{2,f}(Y)$ are irreducible over $K$ and define linearly disjoint extensions, then $C/K$ has maximal wild monodromy $\M/K$ with group $\Gal(\M/K) \simeq Q_8 \times Q_8$.
\end{prop}

\begin{proof} Using Maple, one computes $T_f(Y)$ and reduces it mod $2$. The statement about $T_f(Y)$ follows from Hensel's lemma. 

Let $y$ be a root of $T_f(Y)$. Define $\rho T=S=X-y$ and choose $s_0(y)^{1/2}$ and $s_2(y)^{1/2}$ such that $2s_0(y)^{1/2}s_2(y)^{1/2}=s_1(y)$. Then 
\begin{align*}
f(S+y)&=s_0(y)+s_1(y)S+s_2(y)S^2+s_3(y)S^3+s_4(y)S^4+S^5\\
      &=(s_0(y)^{1/2}+s_2(y)^{1/2}S)^2+s_3(y)S^3+s_4(y)S^4+S^5\\
      &=(s_0(y)^{1/2}+s_2(y)^{1/2}\rho T)^2+s_3(y)\rho^3 T^3+s_4(y)\rho^4 T^4+\rho^5 T^5.
\end{align*}
The change of variables
\begin{equation*}
\rho T=S=X-y \; \; {\rm and } \; \; Y=2W+(s_0(y)^{1/2}+s_2(y)^{1/2}S),
\end{equation*}
induces
\begin{equation*}
W^2+(s_0(y)^{1/2}+s_2(y)^{1/2}S)W=s_3(y)T^3+s_4(y)\rho T^4+ \rho^2 T^5,
\end{equation*}
which is an equation of a quasi-projective flat scheme over $K(y,f(y)^{1/2})^{\circ}$ with special fiber given by $w^2-w=t^3$.

Let $(y_i)_{i=1,\dots,4}$ (resp. $(y_i)_{i=5,\dots,8}$) be the roots of $T_{1,f}(Y)$ (resp. $T_{2,f}(Y)$). Then, for any $i \in \lbrace 1, \dots,4 \rbrace$ and $j \in \lbrace 5, \dots ,8\rbrace$, the above computations show that $C$ has stable reduction over $\L:=K(y_i,y_j,f(y_i)^{1/2},f(y_j)^{1/2})$ (use \cite{Liu} 10.3.44). Moreover two distinct roots of $T_{1,f}(Y)$ (resp. $T_{2,f}(Y)$) induce equivalent Gauss valuations on $K(C)$, else it would contradict the uniqueness of the stable model. In particular,
\begin{align*}
\label{ineg}
v(y_i-y_j)\geq v(\rho),& \; \; i \neq j \in \lbrace 1,2,3,4\rbrace \; \; {\rm or} \; \; i \neq j \in \lbrace 5,6,7,8 \rbrace, \tag{7}\\
v(y_i-y_j)=0, &\; \; i \in \lbrace 1,2,3,4\rbrace \; \; {\rm and} \; \; j \in \lbrace 5,6,7,8 \rbrace,
\end{align*}
which implies that $v(\disc(T_{1,f}(Y)))+v(\disc(T_{2,f}(Y)))=v(\disc(T_f(Y)))$ and $\disc(T_{i,f}(Y)) \geq 12v(\rho)=8v(2)$ for $i=1,2$. These are equalities if and only if \eqref{ineg} are all equalities. Using Maple, one has
\begin{equation*}
2^{-16}\disc(T_f(Y))=b_3^8(1+b_3b_2+b_3^2b_4)^4 \mod 2,
\end{equation*}
thus one gets that \eqref{ineg} are all equalities, therefore 
\begin{equation*}
0, \frac{y_2-y_1}{\rho},\frac{y_3-y_1}{\rho},\frac{y_4-y_1}{\rho},
\end{equation*}
are all distinct mod $\pi_K$. Applying Hensel's lemma to $T_{1,f}(\rho Y+y_1)$ shows that $K(y_1)/K$ is Galois. The same holds for $K(y_5)/K$.

Let's denote by $E_1/k$ and $E_2/k$ (resp. $\infty_1$ and $\infty_2$) the genus $1$ curves in the stable reduction of $C$ (resp. their crossing points with $D_0$). The group $\Aut_k(\mathcal{C}_k)^{\#} \simeq \Aut_{k,\infty_1}(E_1) \times \Aut_{k,\infty_2}(E_2) $ has a unique $2$-Sylow subgroup isomorphic to $\Syl_2(\Aut_{k,\infty_1}(E_1)) \times \Syl_2(\Aut_{k,\infty_2}(E_2)) $ where $\Aut_{k,\infty_i}(E_i) \simeq \Sl_2(\F_3)$ denotes the subgroup of $\Aut_k(E_i)$ leaving $\infty_i$ fixed. 

\indent First, we show that $\L/K$ is the monodromy extension $\M/K$ of $C/K$. Let $\sigma \in \Gal(\L/K)$ inducing the identity on $\mathcal{C}_k/k$. We show that $\forall i \in \lbrace 1, \dots , 8\rbrace$, $\sigma(y_i)=y_i$. Otherwise, since $\sigma$ is an isometry, $\sigma(y_1) \notin \lbrace y_5,\dots, y_8\rbrace$ and we can assume that $\sigma(y_1)=y_2$. So :
\begin{equation*}
\sigma(\frac{X-y_1}{\rho}) = \frac{X-y_1}{\rho}+\frac{y_1-y_2}{\rho},
\end{equation*}
so $\sigma$ does not induce the identity on $\mathcal{C}_k/k$. If $\sigma(s_0(y_i)^{1/2})=-s_0(y_i)^{1/2}$ for some $i$ then $\sigma(s_2(y_i)^{1/2})=-s_2(y_i)^{1/2}$ and
\begin{equation*}
\sigma(W)-W=s_0(y_i)^{1/2}+s_2(y_i)^{1/2} \rho T,
\end{equation*}
therefore $\sigma$ acts non trivially on $\mathcal{C}_k/k$. It implies that for all $ i \in \lbrace 1, \dots , 8 \rbrace$, $\sigma(s_0(y_i)^{1/2})=s_0(y_i)^{1/2}$ and $\sigma=\Id$. Since $\M \subseteq \L$, this shows that $\L = \M$.

Now, we show that the wild monodromy is maximal assuming that $T_{1,f}(Y)$ and $T_{2,f}(Y)$ are irreducible over $K$ and define linearly disjoint extensions. One has natural morphisms :
\begin{equation*}
\Gal(\M/K) \overset{i}{\hookrightarrow} Q_8\times Q_8 \overset{p}{\longrightarrow} Q_8 \times Q_8/{\rm Z}(Q_8)  \overset{q}{\longrightarrow} Q_8/{\rm Z}(Q_8)\times Q_8/{\rm Z}(Q_8).
\end{equation*}
For any $i \in \lbrace 1, \dots ,4 \rbrace$ and $j \in \lbrace 5, \dots, 8 \rbrace$, $(i,j) \neq (1,5)$, there exists $\sigma_{i,j} \in \Gal(\M/K)$ such that :
\begin{equation*}
\sigma_{i,j}(y_1)=y_i \; \; {\rm and} \; \; \sigma_{i,j}(y_5)=y_j,
\end{equation*}
which is seen to act non trivially on $\mathcal{C}_k/k$. The composition $q\circ p \circ i$ is then surjective, it implies that $p \circ i(\Gal(\M/K))$ is a subgroup of $Q_8 \times Q_8/{\rm Z}(Q_8)$ of index at most $2$ so it contains $\Phi(Q_8 \times Q_8/{\rm Z}(Q_8))={\rm Z}(Q_8) \times \lbrace 1 \rbrace = \Ker q$. It follows that $p \circ i$ is onto and $i(\Gal(\M/K))$ is a subgroup of $Q_8 \times Q_8$ of index at most $2$ so it contains $\Phi(Q_8\times Q_8)={\rm Z}(Q_8) \times {\rm Z}(Q_8) \supseteq \Ker p$. Finally, one has $i(\Gal(\M/K)) = Q_8 \times Q_8$.
\end{proof}

\textbf{Example :} Let $f_0(X):=1+2^{3/5}X^2+X^3+2^{2/5}X^4+X^5$ and ${K:=\Q_2^{\rm ur}(2^{\frac{1}{15}})}$. Then, one checks using Magma that $T_{f_0}(Y)=T_{1,f_0}(Y)T_{2,f_0}(Y)$  where $T_{1,f_0}(Y)$ and $T_{2,f_0}(Y)$ are irreducible polynomials over $K$ and $T_{2,f_0}(Y)$ is irreducible over the decomposition field of $T_{1,f_0}(Y)$. So, the curve $C_0/K$ defined by $Y^2=f_0(X)$ has maximal wild monodromy $\M/K$ with group $\Gal(\M/K) \simeq Q_8 \times Q_8$.
\begin{verbatim}
q2:= pAdicField(2,8);
q2x<x>:=PolynomialRing(q2);
k<pi>:=TotallyRamifiedExtension(q2,x^15-2);
K<rho>:=UnramifiedExtension(k,8);
Ky<y>:=PolynomialRing(K);
b3:=1;
b2:=pi^9;
b4:=pi^6;
T:=(2*b2*y+3*b3*y^2+4*b4*y^3+5*y^4)^2-
4*(1+b2*y^2+b3*y^3+b4*y^4+y^5)*(b2+3*b3*y+6*b4*y^2+10*y^3);
F,a,A:=Factorization(T: Extensions:= true);
Degree(F[1][1]);Degree(F[2][1]);
L1:=A[1]`Extension;
L1Y<Y>:=PolynomialAlgebra(L1);
TY:=L1Y!Eltseq(T);
G:=Factorization(TY);
G[1][2];G[2][2];G[3][2];G[4][2];G[5][2];
Degree(G[5][1]);
\end{verbatim}
This has the following consequence for the Inverse Galois Problem :

\begin{coro}
With the notations of the above example, let $(y_i)_{i=1,\dots,8}$ be the roots of $T_{f_0}(Y)$ and $\M=K(y_1,\dots, y_8,f_0(y_1)^{1/2},\dots,f_0(y_8)^{1/2})$. Then $\M/K$ is Galois with Galois group isomorphic to $Q_8 \times Q_8 \simeq \Syl_2(\Aut_k(\mathcal{C}_k)^{\#})$.
\end{coro}

We now give results about the arithmetic of the monodromy extension of the previous example.
\begin{center}
\begin{tikzpicture}[scale=1.5,xscale=1.5]
\node (M) at (0,1) { $\M$};
\draw (-0.1,0.8)--(-0.8,0.2);
\draw (+0.1,0.8)--(+0.8,0.2);
\node (K1) at (-1,0) { $K_1:=K(y_1,f_0(y_1)^{1/2})$};
\node (K2) at (1,0) { $K_2:=K(y_5,f_0(y_5)^{1/2})$};
\draw (-0.1,-0.8)--(-0.8,-0.2);
\draw (+0.1,-0.8)--(+0.8,-0.2);
\node (K) at (0,-1) {$K:=\Q_2^{\rm ur}(2^{1/5})(2^{1/3})$};
\draw (0,-1.2)--(0,-1.8);
\node (Q2) at (0,-2) {$\Q_2^{\rm ur}(2^{1/5})$};
\end{tikzpicture}
\end{center}

\indent First of all, $\M/\Q_2^{\rm ur}(2^{1/5})$ is the monodromy extension of $C_0/\Q_2^{\rm ur}(2^{1/5})$. Indeed, let $\theta$ be a primitive cube root of unity. The curve $C_0$ has a stable model over $\M$ and $\sigma \in \Gal(K/ \Q_2^{\rm ur}(2^{1/5}))$ defined by $\sigma(2^{1/3})=\theta 2^{1/3}$ acts non trivially on the stable reduction by $t \mapsto \overline{\theta} t$ . It implies that 
\begin{equation*}
\Gal(\M/\Q_2^{\rm ur}(2^{1/5})) \hookrightarrow \Sl_2(\F_3) \times \Sl_2(\F_3). 
\end{equation*}
Moreover, $\M/K_1$ is Galois with group isomorphic to $Q_8$. Indeed, from the proof of proposition \ref{proptypeii}, since $\Gal(M/K_1)$ acts trivially on one of the two elliptic curves of the stable reduction, one has the injection :
\begin{equation*}
\Gal(M/K_1) \hookrightarrow Q_8 \times \lbrace \Id \rbrace.
\end{equation*}
 Moreover the image of this injection is mapped onto $Q_8/{\rm Z}(Q_8)$ so $\Gal(M/K_1) \simeq Q_8$. The extensions $K_1/K$ and $K_2/K$ being linearly disjoint. It implies that $\Gal(K_2/K) \simeq Q_8$ and it follows that $\Gal(K_2/\Q_2^{\rm ur}(2^{1/5})) \simeq \Sl_2(\F_3)$. Similarly, $\Gal(K_1/\Q_2^{\rm ur}(2^{1/5})) \simeq \Sl_2(\F_3)$.

 \indent This has consequences on the possible ramification subgroups arising in the filtrations of ${}_{1}G:=\Gal(K_1/K)$ and ${}_{2}G:=\Gal(K_2/K)$. Namely there are no subgroups of order $4$, otherwise there would be a normal subgroup of order $4$ in $ \Sl_2(\F_3)$. So the possible subgroups arising in the ramification filtrations of ${}_{1}G$ and ${}_{2}G$ are $Q_8$, ${\rm Z}(Q_8)$ and $\lbrace 1 \rbrace$.

\indent Using Magma (see \cite{Web}) one computes the lower ramification filtrations
\begin{align*}
{}_{1}G=({}_{1}G)_0=({}_{1}G)_1 & \supsetneq {\rm Z}({}_{1}G) =({}_{1}G)_2=({}_{1}G)_3 \supsetneq \lbrace 1 \rbrace,\\
{}_{2}G=({}_{2}G)_0= \dots =({}_{2}G)_5 & \supsetneq {\rm Z}({}_{2}G) =({}_{2}G)_6= \dots = ({}_{2}G)_{69} \supsetneq \lbrace 1 \rbrace . 
\end{align*}
In order to compute the lower ramification filtration of $\Gal(\M/K)$, we now determine its upper ramification filtration since it enjoys peculiar arithmetic properties. Using lemma 3.5 of \cite{Kid} and the expressions of $\varphi_{K_1/K}$ and $\varphi_{K_2/K}$, one sees that $K_1/K$ and $K_2/K$ are \textsl{arithmetically disjoint}. According to \cite{Yam} theorem 3, one has for any  $u \in \R$ :
\begin{equation*}
\Gal(\M/K)^u \simeq {}_{1}G^u \times {}_{2}G^u.
\end{equation*}
So one gets : 
\[
 \Gal(\M/K)^u \simeq  \left\{ 
				\begin{array}{ll}
				 {}_{1}G\times {}_{2}G,& -1 \leq u \leq 1,\\
                 {\rm Z}({}_{1}G) \times {}_{2}G,& 1 < u \leq 3/2,  \\
                 \lbrace 1 \rbrace \times {}_{2}G,& 3/2 < u \leq 5,\\
				 \lbrace 1 \rbrace \times {\rm Z}({}_{2}G), & 5 < u \leq 21,\\
				 \lbrace 1 \rbrace \times \lbrace 1 \rbrace, & 21 < u.
                \end{array}
           \right.
\]
One deduces the lower ramification filtration of $\Gal(\M/K)$ :
\[
 \Gal(\M/K)_i \simeq  \left\{ 
				\begin{array}{ll}
				 {}_{1}G\times {}_{2}G,& -1 \leq i \leq 1,\\
                 {\rm Z}({}_{1}G) \times {}_{2}G,& 2 \leq i \leq 3,  \\
                 \lbrace 1 \rbrace \times {}_{2}G,& 4 \leq i \leq 31,\\
				 \lbrace 1 \rbrace \times {\rm Z}({}_{2}G), & 32 \leq  i \leq 543,\\
				 \lbrace 1 \rbrace \times \lbrace 1 \rbrace, & 544 \leq  i.
                \end{array}
           \right.
\]
Let denote the genus $1$ irreducible components of $\mathcal{C}_k/k$ by $E_1$ and $E_2$. Let $H_1$ (resp. $H_2$) be a finite subgroup of $\Syl_2(\Aut_{k,\infty_1}(E_1))$ (resp. $\Syl_2(\Aut_{k,\infty_2}(E_2))$) and $\ell \neq 2$ be a prime number. One has :
\begin{equation*}
\Pic^{\circ}(\mathcal{C}_k)[\ell]^{H_1 \times H_2} = \Pic^{\circ}(E_1)[\ell]^{H_1} \times \Pic^{\circ}(E_2)[\ell]^{H_2}.
\end{equation*}
According to lemma \ref{guralnick} one has $\dim_{\F_{\ell}} \Pic^{\circ}(E_i)[\ell]^{H_i} = 2g(E_i/H_i)$. It follows that  $\sw(\Jac(C_0)/K)=45$.\\

\textit{Degeneration type II :}
\begin{prop}\label{proptypeiii}
Let $a^9=2$, $K:=\Q_2^{\rm ur}(a)$, $\rho:=a^4$ and $C/K$ be the smooth, projective, geometrically integral curve birationally given by 
\begin{equation*}
Y^2=f(X)=1+a^3X^2+a^6X^3+X^5.
\end{equation*}
Then, $C$ has stable reduction of type II and $C/K$ has maximal wild monodromy $\M/K$ with group $\Gal(\M/K) \simeq (Q_8\times Q_8)\rtimes \Z/2\Z$.
\end{prop}
\begin{proof}
Using Magma, one determines the Newton polygon of $T_f(Y)$. Then, $T_f(Y)$ has $8$ roots $(y_i)_{i=1,\dots,8}$ of valuation $\frac{7}{24}v(2)$. By considering the Newton polygon of $\Delta(Z)=(T_f(Z+y_1)-T_f(y_1))/Z$, one shows that $\Delta(Z)$ has $3$ roots (say $y_2-y_1$, $y_3-y_1$ and $y_4-y_1$) of valuation $v(\rho)$ and $4$ roots of valuation $v(2)/3$.

Let $y$ be a root of $T_f(Y)$. Define $\rho T=S=X-y$ and choose $s_0(y)^{1/2}$ and $s_2(y)^{1/2}$ such that $2s_0(y)^{1/2}s_2(y)^{1/2}=s_1(y)$. Then the change of variables
\begin{equation*}
\rho T=S=X-y \; \; {\rm and } \; \; Y=2W+(s_0(y)^{1/2}+s_2(y)^{1/2}S),
\end{equation*}
induces
\begin{equation*}
W^2+(s_0(y)^{1/2}+s_2(y)^{1/2}S)W=\frac{s_3(y)\rho^3}{4}T^3+\frac{s_4(y)\rho^4}{4} T^4+ \frac{\rho^5}{4} T^5,
\end{equation*}
which is an equation of a quasi-projective flat scheme over $K(y,f(y)^{1/2})$ with special fiber given by $w^2-w=t^3$. The same argument as in the degeneration type I shows that $C$ has stable reduction of type II over $\L=K(y_1,\dots, y_8,f(y_1)^{1/2},\dots,f(y_8)^{1/2})$.

\indent We first show that $\L/K$ is the monodromy extension $\M/K$ of $C/K$. Let $\sigma \in \Gal(\L/K)$ inducing the identity on $\mathcal{C}_k/k$. We show that $\forall i \in \lbrace 1, \dots , 8\rbrace$, $\sigma(y_i)=y_i$. Else, for example, $\sigma(y_1)=y_2$ or $\sigma(y_1)=y_5$. It follows from the properties of the roots of $\Delta(Z)$ that, if $\sigma(y_1)=y_2$ then $\sigma$ acts by non trivial translation on $\mathcal{C}_k/k$ and if $\sigma(y_1)=y_5$ then $\sigma$ acts on $\mathcal{C}_k/k$ by permuting the genus $1$ components. Once again, the same computations as in the degeneration type I show that $\forall i \in \lbrace 1, \dots,8 \rbrace$, $\sigma(f(y_i)^{1/2})=f(y_i)^{1/2}$. Since $\M \subseteq \L$, one gets $\M = \L$.

\indent Now, we show that the wild monodromy is maximal. Let's consider the canonical morphism :
\begin{equation*}
\Gal(\M/K) \overset{i}{\hookrightarrow} \Syl_2(\Aut_k(\mathcal{C}_k)^{\#}) \simeq (Q_8\times Q_8) \rtimes \Z /2\Z.
\end{equation*}
One sees $Q_8\times Q_8$ as the subgroup $(Q_8\times Q_8)\rtimes \lbrace 1 \rbrace$ of $(Q_8\times Q_8) \rtimes \Z /2\Z$. Set $ H:=i(\Gal(\M/K)) \cap (Q_8 \times Q_8)$. One has natural morphisms :
\begin{equation*}
H \overset{p}{\longrightarrow} Q_8 \times Q_8/{\rm Z}(Q_8) \overset{q}{\longrightarrow} Q_8/{\rm Z}(Q_8)\times Q_8/{\rm Z}(Q_8).
\end{equation*}
Using Magma (see \cite{Web}) one shows that $T_f(Y)$ is irreducible over $K$ and over $K(y_1)$ one has the following decomposition in irreducible factors :
\begin{equation*}
T_f(Y)=\prod_{i=1}^4(Y-y_i)T_2(Y),
\end{equation*}
and $T_2(Y)$ decomposes over $K(y_1,y_5)$. It implies that $q\circ p \circ i$ is surjective and $p(H)$ is a subgroup of index at most $2$ so it contains $\Phi( Q_8\times Q_8/{\rm Z}(Q_8))$ and as for type I, one has $p(H) = Q_8\times Q_8/{\rm Z}(Q_8)$. It implies that $H$ is a subgroup of $Q_8 \times Q_8$ of index at most $2$ and again $H=Q_8 \times Q_8$, that is $Q_8 \times Q_8 \subseteq i(\Gal(\M/K))$. Finally on has a natural morphism :
\begin{equation*}
(Q_8\times Q_8) \rtimes \Z /2\Z \overset{r}{\longrightarrow} (Q_8/{\rm Z}(Q_8) \times Q_8/{\rm Z}(Q_8)) \rtimes \Z /2\Z.
\end{equation*}
The composition $r \circ i$ is surjective since there exist $\sigma, \tau \in \Gal(\M/K)$ such that for $i \in \lbrace 1, \dots ,4 \rbrace$ and $j \in \lbrace 5, \dots ,8 \rbrace$
\begin{align*}
&\sigma(y_1)=y_i \; \; {\rm and} \; \; \sigma(y_5)=y_j,\\
&\tau(y_1)=y_5.
\end{align*}
Since the index of $i(\Gal(\M/K))$ in $(Q_8\times Q_8) \rtimes \Z /2\Z $ is at most $2$,  this group contains $\Phi((Q_8\times Q_8) \rtimes \Z /2\Z) \supseteq \Ker r$ so $i(\Gal(\M/K))=(Q_8\times Q_8) \rtimes \Z /2\Z$.
\end{proof}
Again we derive the following result for the Inverse Galois Problem :
\begin{coro}
With the notations of Proposition \ref{proptypeiii}, let $(y_i)_{i=1,\dots,8}$ be the roots of $T_{f}(Y)$ and $\M=K(y_1,\dots, y_8,f(y_1)^{1/2},\dots,f(y_8)^{1/2})$. Then $\M/K$ is Galois with Galois group isomorphic to $(Q_8 \times Q_8)\rtimes \Z / 2 \Z \simeq \Syl_2(\Aut_k(\mathcal{C}_k)^{\#})$.
\end{coro}

\noindent \textbf{Remark.} Throughout this paper, we have described monodromy extensions as decomposition fields of explicit polynomials being $p$-adic approximations of the so called \textsl{monodromy polynomial} of \cite{LM2}. The point is that the roots of the monodromy polynomial are the centers of the blowing-ups giving the stable reduction of a $p$-cyclic cover of $\Proj_K$ with equidistant geometry. For a given genus, the expression of the monodromy polynomial is somehow generic, making it quite complicated. Since $p$-adically close polynomials with same degrees define the same extensions, it was natural to drop terms having a small $p$-adic contribution in our examples to obtain modified monodromy polynomials easier to handle than the actual monodromy polynomial.

\bibliographystyle{alpha} 
\bibliography{maxmonoutf}

\end{document}